\documentclass[12pt]{article}
\usepackage{todonotes}
\usepackage{amsmath}                            % (1)
\usepackage{amssymb}                            % (1)
\usepackage{amsthm}                             % (2)
\usepackage[pdfborder={0 0 0}]{hyperref}        % (4)
\usepackage{algorithm}
\usepackage{hyperref}
\usepackage[utf8]{inputenc}
\usepackage{enumitem}
\usepackage{centernot}
\usepackage{placeins}
\usepackage{amsthm}
\usepackage{amsmath}
\usepackage{mathrsfs}
\usepackage{amssymb}
\usepackage{scalefnt}
\usepackage{amsfonts}
\usepackage{xcolor}
\usepackage{color}
\usepackage{xspace}
\usepackage{microtype}
\usepackage[mathscr]{eucal}
\usepackage{float}
\usepackage{amsthm}

\usepackage{centernot}
\usepackage{tikz}
\usepackage{placeins}
\usepackage{amsthm}
\usepackage{amsmath}
\usepackage{amssymb}
\usepackage{scalefnt}
\usepackage{amsfonts}
\usepackage{color}
\usepackage{xspace}
\usepackage{microtype}
\usepackage[mathscr]{eucal}

\DeclareMathOperator{\Seq}{\subseteq}
\DeclareMathOperator{\card}{card}

\DeclareMathOperator{\cone}{\textbf{cone}}
\DeclareMathOperator{\cor}{\textbf{cor}}
\DeclareMathOperator{\rev}{\textbf{rev}}
\DeclareMathOperator{\qface}{\textbf{qface}}
\DeclareMathOperator{\succc}{\textbf{succ}}
\DeclareMathOperator{\head}{\textbf{head}}
\DeclareMathOperator{\tail}{\textbf{tail}}
\DeclareMathOperator{\entry}{\textbf{entry}}
\DeclareMathOperator{\lefttt}{\textbf{left}}
\DeclareMathOperator{\xcor}{\textbf{xcor}}
\DeclareMathOperator{\ycor}{\textbf{ycor}}
\DeclareMathOperator{\rep}{\textbf{reposition}}
\newtheorem{theorem}{Theorem}%[section]
\newtheorem{lemma} [theorem] {Lemma}

\theoremstyle{theorem}

\theoremstyle{definition}
\newtheorem{definition}{Definition}
\usepackage{algorithm}
\usepackage{algpseudocode}
\def\BState{\State\hskip-\ALG@thistlm}
\makeatother
\newcounter{claimCount}
\theoremstyle{remark}

\begin{document}
\title{Traversal with Enumeration of Geometric Graphs in Bounded Space}  	

\author{
Sahand Khakabimamaghani\\
\texttt{sahandk@sfu.ca}
\and
Masood Masjoody\\
\texttt{mmasjood@sfu.ca}
\and
Ladislav Stacho\\
\texttt{lstacho@sfu.ca}
}  	

\maketitle

\begin{abstract}
In this paper, we provide an algorithm for traversing geometric graphs which visits all vertices, and reports every vertex and edge exactly once. To achieve this, we combine a given geometric graph $G$ with the integer lattice, seen as a graph, in such a way that the resulting hypothetical graph can be traversed using the algorithm in \cite{Chavez}. To overcome the problem with hypothetical vertices and edges, we develop an algorithm for visiting any $k$th neighborhood of a vertex in a graph straight-line drawn in the plane using $O(\log k)$ memory. The memory needed to complete the traversal of a geometric graph then turns out to depend on the maximum ratio of the graph distance and Euclidean distance for pairs of distinct vertices of $G$ at Euclidean distance greater than one and less than $2\sqrt{2}$.
\end{abstract} 	

\section{Introduction}\label{sec:intro}
The problem of \emph{graph traversal} is one of the fundamental combinatorial problems. Given as an input an undirected graph $G$ and a vertex $v$, the graph traversal problem is to start on $v$ and continue visiting vertices of $G$ by a sequence of moves along edges of $G$ in such a way that at some point every vertex of $G$ reachable by a path from $v$  is visited at least once. The moves are guided by an edge labelling; every vertex $v$ has the edges incident to it labelled with integers $1$ to $d(v)$ so that labels are a permutation. So a traversal algorithm at each vertex will choose a label of an edge and then moves along the edge with that label to the neighbor of the vertex. A problem related to graph traversal is $s,t$\emph{-connectivity} where we are given two vertices $s$ and $t$ of $G$, and need to decide whether or not there is a path from $s$ to $t$ in $G$. It is obvious that a solution to the graph traversal problem also gives a solution to the $s,t$-connectivity problem. Many basic graph algorithms involve making traversal or determining connectivity.  	

The time complexity of both problems is well understood and it is linear in the number of edges of the graph. This can be achieved by classical breadth-first search or depth-first search algorithms. Note that these algorithms can solve the two problems also in their directed version, i.e., when the graph is a directed graph. The space required to run these algorithms is linear as well, and understanding the space complexity of these problems was a major problem for several decades.  	

It was long known \cite{aleliunas79} that a random walk (a walk that starts at a vertex and chooses subsequent vertices uniformly at random from current available neighbors) will visit all the vertices reachable from the original vertex in polynomial number of steps. The major problem was to derandomize this simple algorithm without substantially increasing the space.  	

The first deterministic super-polynomial time algorithm for the traversal problem is the algorithm by Savitch \cite{savitch70} that needs $O(\log^2 n)$ space.  The first improvement on this classical result was done by Nisan et al. \cite{nisan92} by showing that the traversal can be performed in $O(\log^{3/2}n)$ space (the time is still super-polynomial). Building on this work as well as on related works \cite{nisan92a,saks99}, Armoni et al. \cite{armoni00} improved the space complexity to $O(\log^{4/3}n)$. Note that the most space-efficient polynomial time algorithm for the traversal problem was Nisan's \cite{nisan92b} which required $O(\log^2 n)$ space. A big improvement was achieved by Trifonov \cite{trifonov05} whose polynomial algorithm required $O(\log n\log\log n)$ space only. Note at this point that $\Omega(\log n)$ space is needed for any algorithm to solve the traversal problem. In 2005, Reingold in his seminal work \cite{reingold05}  solved the long standing problem by presenting an $O(\log n)$ space deterministic polynomial algorithm that solves the traversal problem on any $n$-vertex graph.  	

The traversal problem asks to visit every vertex at least once, and in fact in all the algorithms mentioned above a vertex may be visited many times.  In many applications one needs to have a list of vertices, edges, or some other combinatorial objects defined on graphs. For example, in Kruskal's minimum cost spanning tree algorithm, one needs an ordered list of edges by their weights. In the following we describe work on traversal algorithms which provide such lists. We will refer to such algorithms as \emph{traversal with enumeration} algorithms.  Note at this point that the edge labelings in the traversal problem considered above are arbitrary. Also it is easy to see that the  breadth-first search and depth-first search algorithms can be adapted to traversals with enumeration. However, these algorithms are linear in space  as we mentioned already. We are interested in  most space-efficient traversal with enumeration algorithms. The lists produced by such algorithms are write-only for the algorithms and contain every vertex/edge exactly once. The information provided in graphs considered just as purely combinatorial objects is not enough to achieve this in the same $O(\log n)$ space as the original traversal algorithms. Instead, the research concentrated on graphs embedded in a space and, hence, considered them as geometric objects.  	

One of the first traversal algorithms with enumeration was described in \cite{gold1977}, which traverses triangulations of the plane. The main idea is used in all subsequent works, so we outline it here. One uses the geometric information from the embedding as part of the input to the traversal algorithm and using this information orders edges of the triangulation. The trick is that the query to compare two given edges can be locally computed in $O(\log n)$ space. Then each face will have an ``entering edge" (say, first in the order of the three edges) which will be used to traverse the triangulation as a tree. The embedding of the triangulation is given using the rotation system (a cyclic ordering of edges in the embedding  of the graph given at every vertex). These orderings are then used to traverse along faces using $O(\log n)$ space and search for the minimum edges, move from face to face, etc. Therefore such traversal can be performed in $O(\log n)$ space. Moreover, the geometric information can be further used to determine exact time when to output a visited vertex/edge to the list. In \cite{Avis1991}, authors describe a traversal algorithm with enumeration that can traverse all arrangements of convex polytopes. In \cite{Berg}, the algorithm is extended to all planar subdivisions and its running time is improved in \cite{bose}. All of the works in \cite{Avis1991,bose,Berg} consider plane subdivisions. First result on traversal with enumeration that extends to nonplanar graphs has been presented in \cite{Chavez}. The authors define a, so called, \emph{quasi-planar subdivision} as a graph (straight-line) embedded into the plane, whose vertex set can be partitioned into two sets $V_p$ and $V_c$, vertices in $V_p$ induce a plane graph $P$ (backbone), the outer-face of $P$ does not contain any vertex of $V_c$, and no edge of $P$ is crossed by any other edge. One can imagine a quasi-planar graph as a plane graph in which every face may contain an arbitrary graph which joins only to the vertices of the face. A quasi-planar subdivision is said to satisfy the \emph{left-neighbor rule} if every vertex of a subgraph inside of a face of its backbone has a neighbor on the face that is to the left of the vertex or has a smaller $x$-coordinate; see Figure \ref{quasiPL} for an example. It was proved in \cite{Chavez} that every quasi-planar subdivision that satisfies the left-neighbor rule can be traversed with enumeration in $O(\log n)$ space and polytime.
\begin{figure}[h]	

\begin{center}
\includegraphics[scale=0.65, angle=0]{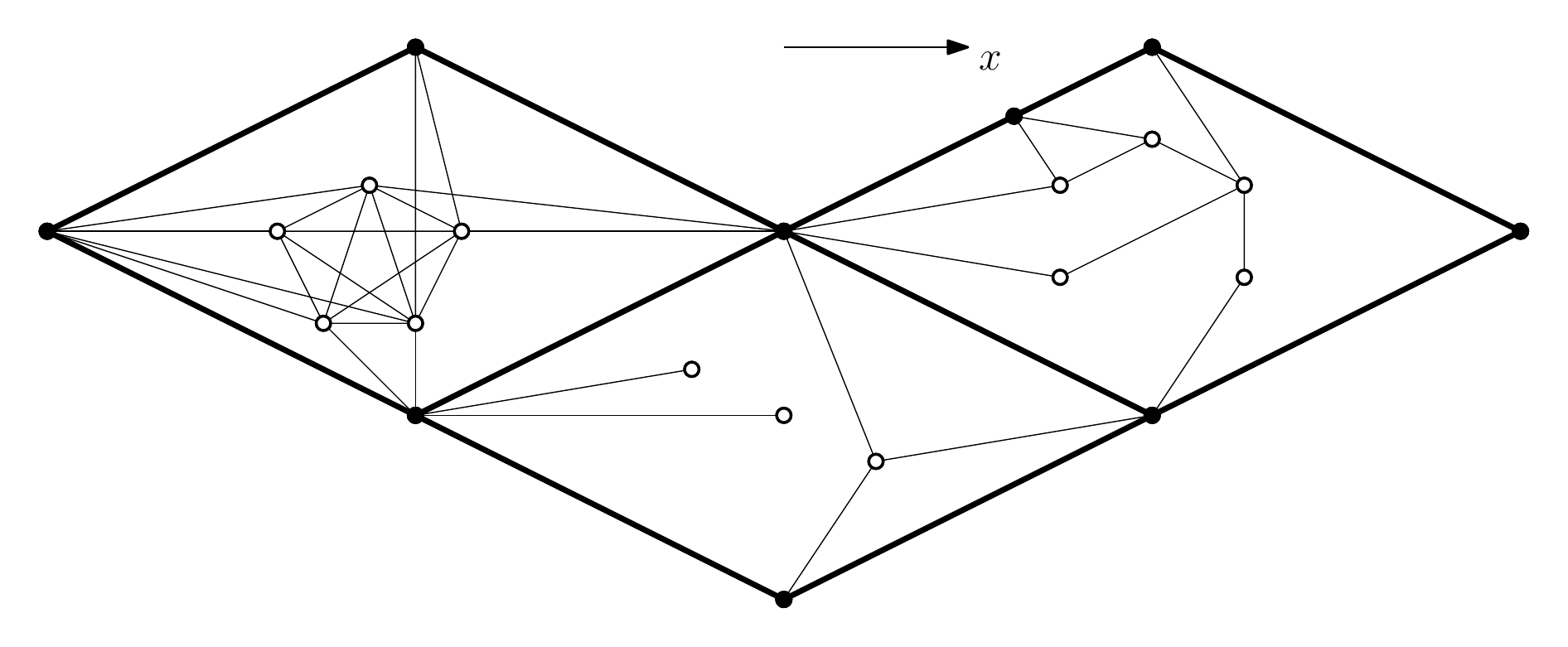}
\caption{\cite{Chavez} An example of a quasi-planar subdivision satisfying the left-neighbour rule. The bold edges and filled vertices are those of the underlying planar subgraph.}
\label{quasiPL}
\end{center}

\end{figure}

Note at this point that traversal algorithms are often modelled using so called walking (or jumping) automata on graphs \cite{Cook}. Such an automaton has a set of states, and a number of pebbles that represent vertex names and are used to mark certain vertices temporarily. The pebbles can be moved from vertex to adjacent vertex (“walk”) or directly to a vertex containing another pebble (“jump”). The position of the automaton on the graph is itself marked by a pebble. Thus, walking represents replacing a vertex name by some adjacent vertex found in the input, and jumping represents copying a previously recorded vertex name. It follows that the space required by such an automaton depends on the number of pebbles used, and obviously a walking automaton that uses  $c$ pebbles and can traverse a graph corresponds to an  $O(c\log n)$ space  traversal algorithm. In this paper we will model our traversal algorithm using a walking automaton.  	

We adapt the traversal algorithm from \cite{Chavez} to a traversal algorithm for geometric graphs which reports every vertex and every edge exactly once. To achieve this, we develop a walking automaton for local exploration of graphs straight-line drawn in the plane (Section~\ref{explore-N-k}). This automaton will visit all vertices within graph distance $k$ (for any fixed $k\in\mathbb{N}$) from a given vertex, and will use as few as $k+1$ pebbles, i.e., its space complexity is $O(k\log n)$.  	

%\section{Definitions and Notations}\label{intro}\todo{the title should change-not everything here is a definition}
%\todo{material from here to the end of the section used to be under a separate section Definitions and Notations}
\begin{definition}[basic notions]\
\begin{itemize}
\item Let $G=(V,E)$ be a graph. For every $v\in V$ we denoted the set of neighbors of $v$ in $G$ by $N_G(v)$ and define $N_G[v]=N_G(v)\cup \{v\}$. Moreover, for every $A\subseteq V$ we define $N_G[A]=\bigcup_{v\in A}N_G[v]$ and set $N_G(A)=N_G[A]\setminus A$.
\item A geometric graph is a (drawing of a) graph having a finite subset $V$ of the plane as its vertex set, and whose edge set consists of line-segments between all pairs of distinct points in $V$ with Euclidean distance $\le 1$. 
\item For every connected geometric graph $G$ of order $\ge 2$ we define 
\begin{equation*}\label{norm.of.G}
r(G)=\min_{\forall x,y\in V(G)} \{M:  1< d_E(x,y)< 2\sqrt 2 \Rightarrow d_G(x,y)\le M d_E(x,y)\},
\end{equation*} 
where $d_G$ and $d_E$ denote the graph distance in $G$ and the Euclidean distance in the plane, respectively.
\end{itemize}
\end{definition}
\noindent We shall use the invariant $r(G)$ to bound the space complexity of a traversal with enumeration algorithm for $G$. In particular, we provide such an algorithm that uses $\lfloor 2\sqrt{2} r(G)\rfloor$ pebbles. We call a geometric graph $G$ \emph{well-embedded} if $r(G)$ is a constant. Hence as  a consequence we obtain a traversal with enumeration algorithm with $O(\log n)$ space complexity for well-embedded geometric graphs.  	

The main idea behind such algorithm is to obtain from $G$, which in general can have many crossings, a quasi-planar graph satisfying the left-neighbor rule. We achieve this by fixing a Cartesian coordinate system and assuming, with no loss of generality (since $G$ is finite), that no vertex of $G$ has an integer coordinate and no edge of $G$ passes through a point in $\mathbb{Z}\times \mathbb{Z}$. Then we create a virtual graph which is an augmentation of $G$ and the grid graph given by the integer lattice. For traversing the resulting graph, we use an adaptation of the traversal algorithm in \cite{Chavez} combined with an algorithm (Algorithm \ref{k-th.neigborhood}) for exploring the $k$th neighborhood of a vertex in a graph straight-line drawn  in the plane. In the following section, we present the formal construction of the virtual graph.  	

\section{The Virtual Graph}\label{sec:virtual_graph}  	

In this section we provide details how to augment a geometric graph $G$ with a grid graph defined on the points of a sub-lattice of the plane.  	

\begin{definition}
For every graph $G$ and $k\in\mathbb{N}$, we define the $k$th neighborhood of a vertex $v$ in $G$ by
\begin{equation}
N_G^k[v]=\{u\in V(G): d_G(u,v)\le k\}.
\end{equation}
(We shall drop the subscript $G$ in $N_G^k[v]$ when $G$ is understood from the context.)
\end{definition}   	

Given a connected geometric graph $G$, we shall regard the infinite integral lattice in the plane as an infinite plane graph and denote it by $L$.
For every $v\in V(G)$ we define the {\em square of} $v$, denoted $s_v$, to be the face of $L$ that encloses $v$, and $\cor(s_v)$ to be the lower left corner of $s_v$.	

We combine $G$ and $L$ into a new graph $G\boxdot L$ as follows. Let $L_G$ be the smallest subgraph of $L$ whose face set contains every $s_v$, $v\in V(G)$.  Let $C(L,G)$ be the set of all edges of $G$ that cross an edge of $L$, and for every vertex $v\in V(G)$ let $\ell_v$ be the line-segment joining $v$ to $\cor(s_v)$. We now define the graph $G\boxdot L$ by
\begin{equation}
V(G\boxdot L)=V(G)\cup V(L_G),
\end{equation}
and
\begin{equation}
E(G\boxdot L)=(E(G)\setminus C(L,G))\cup E(L_G)\cup\{\ell_v:v\in V(G)\}.
\end{equation}  	

\begin{figure}[h]	

\begin{center}
\includegraphics[scale=0.7]{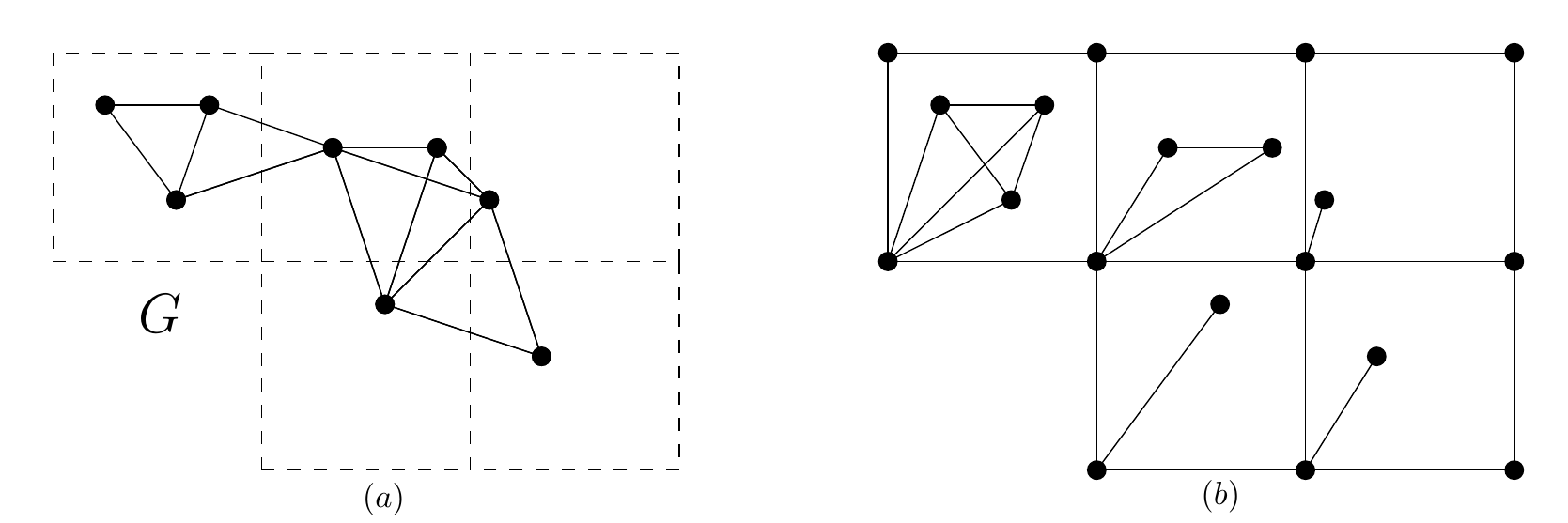}
\caption{(a): A geometric graph $G$ in the graph $L_G$, (b): The quasi-planar graph $G\boxdot L$.}
\label{fig: geo-to-qp}
\end{center}

\end{figure}  	

See Figure \ref{fig: geo-to-qp} for an example. Again, as $G$ is finite we may assume that no newly added edge $\ell_v$ passes through a vertex of $G$. Then, as the planar graph $L_G$ is an induced subgraph of $G\boxdot L$ and no edge of $L$ is crossed by any edge of $G$, it follows that $G\boxdot L$ is a quasi-planar graph with the underlying plane graph $L_G$. Moreover, as every vertex $v$ of $G$ is incident, via $\ell_v$, to a vertex of $L_G$ having a smaller $x$ coordinate than that of $v$, $G\boxdot L$ also satisfies the left-neighbor rule. The only problem with applying the quasi-planar subdivision traversal algorithm from \cite{Chavez} is that $G\boxdot L$ contains hypothetical vertices and edges which cannot be visited or used by a traversing agent. But whenever $r(G)$ (or an upper bound for it) is known, this problem will be resolved. The idea is that the \emph{traversing agent} on $G$, denoted $R^r$, will simulate a \emph{virtual agent}, denoted $R^v$, that will traverse $G\boxdot L$. To be able to perform this simulation in space bounded by $r(G)$, the traversing agent and the virtual agent must maintain a constant graph distance from each other, which turns out to be $\lfloor 2\sqrt{2}r(G)\rfloor$. Then, the local explorations by the virtual agent can be simulated by traversing agent using $\lfloor 2\sqrt{2}r(G)\rfloor$ pebbles. The virtual agent needs to know only a local portion of $G\boxdot L$ . Before each computation step of $R^v$ (e.g., a test to determine the next edge in the rotation system) the traversing agent $R^r$ explores a local portion of $G$ to provide the next arc of $G\boxdot L$ that $R^v$ has to test to determine the next step (e.g., moving along an edge). Determination of such an arc depends not only on the current position of $R^v$ but also on the last edge of $G\boxdot L$ which was traversed by $R^v$.  If $R^v$ is on a vertex of $G$, then $R^r$ starts exploration of the graph starting from that vertex. Moreover, in any step of the game if $R^v$ is moving on an edge of $G$, $R^r$ will make the same move. An algorithm for an exploration used by $R^v$ is described in Section \ref{explore-N-k} in a more general settings.% Indeed, not only for geometric graphs but for all graphs straight-line drawn in the plane, and for every $k\in\mathbb{N}$ we provide an algorithm that can visit vertices within $k$-neighbourhood of every vertex using as few as $k$ pebbles.   	

%graph the traversing robot is capable of exploring a constant graph distance  traversing We shall show that this problem can be resolved by using a robot which alongside performing the quasi-planar traversal \cite{} is capable of is capable of not only traversing $G\boxdot L$ but also exploring $\lfloor 2\sqrt{2}\|\mathscr{G}\|\rfloor$-neighbourhood of any vertex of $G$   	

In the description of our algorithms each edge $e=uv$ is associated with two oppositely directed arcs $\mathbf{e}=(u,v)$ and its {\em reverse}, $\textbf{rev}(\mathbf{e}):=(v,u)$. As $G$ is assumed to be a graph straight-line drawn into the plane, we can define the {\em successor} of an arc $\mathbf{e}=(u,v)$, denoted $\textbf{succ}(\mathbf{e})$, as the first arc $(v,w)$ succeeding $\rev(\mathbf{e})$ counterclockwise around $v$. The {\em predecessor} of $\mathbf{e}$, denoted $\textbf{pred}(\mathbf{e})$, is the arc $\mathbf{e'}$ such that $\textbf{succ}(\mathbf{e'})=\mathbf{e}$. Furthermore, for every vertex (point) $v$ in the plane, we denote its $x$ and $y$ coordinates by $\xcor(v)$ and $\ycor(v)$, respectively. We shall also use the notation of the cone defined by a triple of the points in the plane, as follows. Suppose $a,b,$ and $c$ are points in the plane. We denote by $\angle (a,b,c)$ the counter-clockwise angle with apex $b$ from the ray $ba$ to the ray $bc$. Then, $\cone(a,b,c)$ is the cone with apex $b$ and interior angle $\angle(a,b,c)$, including the supporting ray passing through $a$ and not the one through $c$, see \cite{Chavez}. 	

\section{Local Exploration of Graphs Straight-Line Drawn in the Plane}\label{explore-N-k}   
Let $v$ be any vertex in a straight-line drawn graph $G$. Given $k\in\mathbb{N}$ we want to visit the $k$th neighborhood $N_G^{k}[v]$ of $v$. We shall follow the depth-first search tree of $G$ to distribute a set of $k+1$ pebbles (initially located at $v$) among vertices of the paths of the DFS tree that start from the root $v$, only backtracking when we reach a vertex in $N^k_G(v)$ or when every candidate for a ``new" vertex is adjacent to an already pebbled vertex. Here is a more detailed description of the algorithm. 

\noindent Here is a more detailed description of the algorithm. Consider a graph $G$ with a given straight-line drawing in the plane. On input 

%Given any vertex $v$ in a straight-line drawn graph $G$ and given $k\in\mathbb{N}$ we want a straight-line drawn graph $G$We follow the depth first search tree of a graph straight-line drawn in the plane to explore (or visit) every vertex in the $k$th neighborhood ($k\in \mathbb{N}$) of any given vertex $v$ of the graph. To this end, we shall distribute a set of $k+1$ pebbles (initially located at $v$) among vertices of the paths of the DFS tree that start from the root $v$, only backtracking when we reach a vertex in $N^k_G(v)$ or when every candidate for a ``new" vertex is adjacent to an already pebbled vertex. Here is a more detailed description of the algorithm. % backtracking toward $v$ to continue the search along a ne pathwbacktracking when emanate from $v$ vertices branch of the DFS tree a expand vertxgenerally speaking, our algorithm visits the

\begin{itemize}
\item any $v\in V(G)$, and
\item $k+1$ pebbles initially located at $v$,
%\item the arc $\mathbf{e_0}$ with $\tail(\mathbf{e_0})=v$ having the smallest counterclockwise angle with the $x$-axis,
\end{itemize}
% any $v\in V(G)$, $k+1$ pebbles initially located at $v$, and the arc $\mathbf{e_0}$ with $\tail(\mathbf{e_0})=v$ having the smallest counterclockwise angle with the $x$-axis, 
our algorithm {\em visits} every vertex in $N^k_G[v]$ at least once and does not visit any vertex outside $N^k_G[v]$. Visiting a vertex is defined by moving at least one pebble from a neighboring pebbled vertex to this vertex (by which this vertex will be considered visited then). We move pebbles around in $N^k_G[v]$ using the operators $\succc$ and $\rev$, in a depth-first-search fashion.   	

At any stage (or iteration) the algorithm will be either in \emph{forward} (fw) or \emph{backward} (bw) {\em mode}, and on an arc $\mathbf{e}$ of $G$.  Thus, at any time we can define the {\em state} $(P,\mathbf{e},mode)$, where $P$ is the set of currently pebbled vertices (i.e., vertices where the algorithm stores at least one pebble), $\mathbf{e}$ is the current arc, and $mode\in\{\text{fw},\text{bw}\}$. The computation of the algorithm can be described as a sequence of such states.

Exploration in our algorithm can begin on any arc of $G$ with tail $v$. But for definiteness we set $\mathbf{e_0}$ as the arc with $\tail(\mathbf{e_0})=v$ having the smallest counterclockwise angle with the $x$-axis and set the initial state to be $(\{\tail(\mathbf{e_0}),\head(\mathbf{e_0})\},\mathbf{e_0},\text{fw})$.	

In our algorithm, changing from a state $(P,\mathbf{e},mode)$ to $(P',\mathbf{e'},mode')$ depends on $mode$, adjacency of vertices in $N(\head(\mathbf{e}))$ to the other pebbled vertices, and whether or not $\card(P)=k+1$. Moreover, in such a change of states we will always have $\card(P\triangle P')=1$ and $\head(\mathbf{e})=\tail(\mathbf{e'})$.   	

To visit all vertices in the $k$th neighborhood of $v$, we choose the arc $\mathbf{e_0}$ starting at $v$ as we described above, and initially put all pebbles at $v$. The algorithm then starts moving pebbles in a DFS way, and in such a way that at any time, the pebbled vertices form the vertex set of an induced path in $G$. This  condition is essentially maintained by Steps 10-14 of the algorithm in the forward mode, and by Steps 29-33 in the backward mode. By Steps 10-14, when in a state $(P,\mathbf{e},\text{fw})$, $\tail(\mathbf{e})$ is the only pebbled vertex which is adjacent to $\head(\mathbf{e})$. Furthermore, the mode in the next iteration remains forward only if there is an arc succeeding $\mathbf{e}$ whose head is an unpebbled vertex with no pebbled neighboring vertex except $\head(\mathbf{e})$. If no such arc exist, the algorithms backtracks along $\rev({\mathbf{e}})$ and the following state will be $(P\setminus\{\head(\mathbf{e})\}, \rev({\mathbf{e}}),\text{bw})$ (Steps 15-18). Transition from a state with the backward mode is similar (Steps 29-33 and 36-38). The algorithm terminates when it returns to the arc $\mathbf{e_0}$, by which time all of the vertices in the $k$th neighborhood of $v$ will have been visited. We establish this fact, among other things, in Theorem \ref{thm: k-th-neighbourhood}.

% that is the final state since .%  the successor of the current arc $(4,1)$ is given  algorithm  on input vertex 1. T and produces 

% The algorithm recognizes $\mathbf{e_0}$ as the arc incident to $v$ with the smallest counterclockwise angle.

%\todo{You should show all states, for example the inital state, or the state when we explore all arcs from vertex 2 and moving all pebbles back to 1 - you aren ot showing this state. Or explain in the description which states are listed in the example.}
\begin{figure}[H]	

\begin{center}
\includegraphics[scale=0.77]{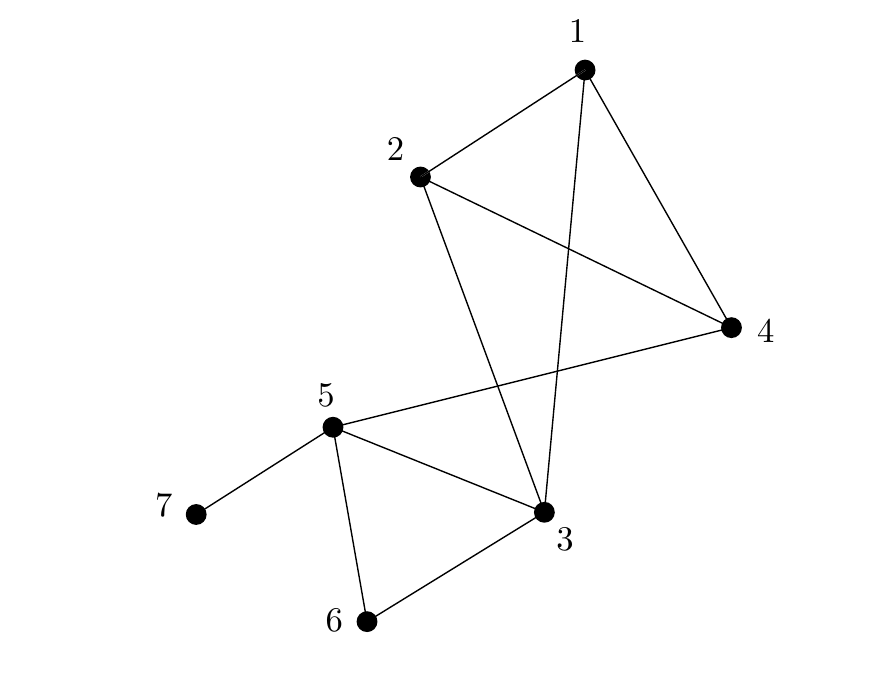}
\caption{An instance of Algorithm \ref{k-th.neigborhood} on input vertex 1 with the task of exploring the second neighborhood of vertex 1 that produces a list of 12 states before it terminates.}% The list of states produced by the algorithm is given.states produces the following list of states:} %by the algorithm are The algorithm first recognizes $\mathbf{e}_0=(1,2)$. To explore the second neighbourhood of vertex 1, the algorithm produces the following list of states until it terminates after state 12 by putting all the pebbles on $v$ (Step 34)}
\begin{tabular}{lll}\small
$1:(\{1,2\},(1,2),\text{fw})$  & \small$2:(\{1\},(2,1),\text{bw})$ & \small$3:(\{1,3\},(1,3),\text{fw})$\\
\small$4:(\{1,3,5\},(3,5),\text{fw})$ &\small $5: (\{1,3\},(5,3),\text{bw})$ &\small $6:(\{1,3,6\},(3,6),\text{fw})$\\
\small$7: (\{1,3\},(6,3),\text{bw})$ & \small$8:(\{1\},(3,1),\text{bw})$ &        \small $9:(\{1,4\},(1,4),\text{fw})$\\
\small$10: (\{1,4,5\},(4,5),\text{fw})$ \small& $11: (\{1,4\},(5,4),\text{bw})$ &\small $12: (\{1\},(4,1),\text{bw})$
\end{tabular}\normalsize
\label{k-explor}
\end{center}

\end{figure}	
\noindent Figure \ref{k-explor} shows an instance of Algorithm \ref{k-th.neigborhood} for exploring the second neighborhood of vertex 1 in the given graph. The algorithm first recognizes $\mathbf{e}_0=(1,2)$ as the current arc in the initial state, and places all pebbles but one on $\head(\mathbf{e}_0)$. Starting from the initial state $(\{1,2\},(1,2),\text{fw})$, it eventually reaches the state $(\{1\},(4,1),\text{bw})$ where all the pebbles are back on vertex $1$. Since $\mathbf{e}_0=\succc((4,1))$ that is the terminal state, according to Step 34 in Algorithm \ref{k-th.neigborhood}.

%\todo{While loop as in Step 11 does not seem correct to me. Maybe you can write is as a separate function, say isucc(e) and describe in words what it returns: first arc e' in counterclockwise order from rev(e) for which head(e') is not pebbled and is not a neighbor of a pebbled vertex different from head(e). If such earc does not exist rev(e) is returned. Something like this. Then you can replace both while loops just by a call to this function.} 

\begin{algorithm}[H]
\caption{Visiting all vertices in $N^k[v]$}\label{k-th.neigborhood}
\textbf{Input:} a straight-line drawing of a graph $G$ in the plane, a vertex $v\in V(G)$ with $k+1$ pebbles on $v$  	

\textbf{Objective:} Visit all vertices in $N^k[v]$
\begin{algorithmic}[1]
\State denote the set of currently pebbled vertices by $P$.
\State let $\mathbf{e_0}$ be the arc with $v=\tail(\mathbf{e_0})$ and the smallest counterclockwise angle with $x$-axis
\State $mode\gets \text{fw}$
\State\label{a1:4} Keep one pebble at $\tail(\mathbf{e_0})$ and move the rest to  $\head(\mathbf{e})$
\State\label{a1:5}$P=\{\head(\mathbf{e_0}),\tail(\mathbf{e_0})\}$
\State $\mathbf{e}\gets \mathbf{e_0}$
\Repeat
\If {$mode=\text{fw}$}
\If {$|P|<k+1$}	

\State $\mathbf{e'}\gets \succc(\mathbf{e})$
\While {$\head(\mathbf{e'})\in (N(P\setminus\{\head(\mathbf{e})\})\setminus P)$}
\State $e'\gets\succc(\rev(\mathbf{e'}))$
\EndWhile
\State $\mathbf{e}\gets \mathbf{e'}$
%\State let $w=\head(\mathbf{e})$ \todo{remove?}
\If {$\head(\mathbf{e})\in P$}
\State move all pebbles on $\tail(\mathbf{e})$ to $\head(\mathbf{e})$
\State $P=P\setminus\{\tail(\mathbf{e})\}$
\State $mode\gets\text{bw}$ 
\Else
\State move all but one of the pebbles on $\tail(\mathbf{e})$ to $\head(\mathbf{e})$
\State $P=P\cup\{\head(\mathbf{e})\}$
\EndIf
\Else

\algstore{traversal}
\end{algorithmic}
\end{algorithm}
\clearpage
\begin{algorithm} 	

\caption{Visit all vertices in $N^k[v]$ (continued)}\label{k-th.neigborhood2}
\begin{algorithmic}[1]
\algrestore{traversal}
\State move all pebbles from $\head(\mathbf{e})$ to $\tail(\mathbf{e})$
\State $\mathbf{e}\gets\rev(\mathbf{e})$
\State $mode\gets \text{bw}$
\EndIf
\Else \Comment{When $mode=\text{bw}$}
\State $\mathbf{e'}\gets \succc(\mathbf{e})$
\While {$\head(\mathbf{e'})\in (N(P\setminus\{\head(\mathbf{e})\})\setminus P)$}
\State $\mathbf{e'}\gets\succc(\rev(\mathbf{e'}))$
\EndWhile
\State $\mathbf{e}\gets \mathbf{e'}$
%\State let $z=\head(\mathbf{e})$
\If {$\mathbf{e}=\mathbf{e_0}$} move all pebbles on $v$ and terminate.
\EndIf
\If {$\head(\mathbf{e})\in P$} 	

\State move all pebbles on $\tail(\mathbf{e})$ to $\head(\mathbf{e})$
\State $P=P\setminus\{\tail(\mathbf{e})\}$
\Else 
\State move all but one of the pebbles on $\tail(\mathbf{e})$ to $\head(\mathbf{e})$
\State $P=P\cup\{\head(\mathbf{e})\}$
\State $mode=\text{fw}$
\EndIf
\EndIf
\Until {$\mathbf{e}=\mathbf{e_0}$} \label{a1:43} 
\end{algorithmic}
\end{algorithm}
%\begin{definition}
%At any time during an application of Algorithm \ref{k-th.neigborhood} we call the triple $(P,\mathbf{e},mode)$ the {\em state} of the algorithm at that time. If the initial state at iteration $0$ is $(\{v\},\mathbf{e_0},fw)$, we denote the state in every $i$th iteration by $state(i)=(P_i,\mathbf{e_i},mode_i)$.
%\end{definition}
\begin{theorem}\label{thm: k-th-neighbourhood}
Let $G$ be graph straight-line drawn in the plane, $v\in V(G)$, and $\mathbf{e_0}$ an arc of $G$ with $\tail(\mathbf{e_0})=v$. Then the following holds:  
\begin{enumerate}
\item[(a)] For every iteration $i$,%\todo{I think you should move definition of P_i here. It is in the proof now.} $P_i$ induces a path in $G$ between $v$ and $\head(\mathbf{e_i})$.
\item [(b)]The states at different times are all distinct.
\item [(c)]The algorithm terminates after a finite number of iterations.
\item [(d)]For every $u\in N^k[v]$ there is an iteration $i$ such that $u\in P_i$. 
\end{enumerate}
\end{theorem}   	

\begin{proof}
We denote the state in every $i$th %\todo{Maybe you should define iteration more formally.} 
iteration ($i\in\mathbb{N}\cup\{0\}$) by $state(i)=(P_i,\mathbf{e_i},mode_i)$, where $i=0$ corresponds to the initial state, so that $state(0)=(\{v\},\mathbf{e_0},\text{fw})$. \textbf{(a) }We use induction on $i$. The claim holds for $i=0$ and $1$, since $G[P_0]$ is a single vertex, and $G[P_1]$ is a 2-path between $v$ and $\head(\mathbf{e_0})$. Suppose $j\in \mathbb{N}$ and the claim holds for every iteration $i<j$. In the case $mode_j=\text{fw}$, $\head(\mathbf{e_{j-1}})=\tail(\mathbf{e_j})$ is the only vertex in $P_{j-1}$ which is adjacent to $\head(\mathbf{e_j})$. Hence, by induction hypothesis, $G[P_{j}]$ is an induced path between $v$ and $\head(e_j)$. On the other hand, if $mode_j=\text{bw}$ then $P_j:=P_{j-1}\setminus\head(\mathbf{e_{j-1}})$ where $\head(e_{j-1})=\tail(e_{j})$. %\mathbf{e_j}=\rev(\mathbf{e_{j-1}})$.
Hence,  $G[P_j]$ is the sub-path of $G[P_{j-1}]$ between $v$ and $\head(\mathbf{e_j})$ and the claim follows by the induction hypothesis.\textbf{(b) }Suppose by the way of contradiction, that there is a repeated state and let $state_*:=(P_*,\mathbf{e_*},mode_*)$ be the first state that occurs at least twice during an application of the algorithm. Let iterations $i$ and $j$ be the first two iterations whose states are equal to $state_*$. Note that because of the termination criterion (Step 34), $state(0)$ cannot be repeated; thereby, $i,j>0$. It is not difficult to observe that under our assumptions and due to how the algorithm extends paths, we have
\begin{equation*}
\head(\mathbf{e_{i-1}})=\head(\mathbf{e_{j-1}})=\tail(\mathbf{e_*}).
\end{equation*}Let

\begin{equation*}
\mathbf{e_*}=(a,b)\quad\&\quad \mathbf{e_k}=(c_k,a)\quad (\text{for }k\in\{i-1,j-1\}).
\end{equation*}
Then, by the choice of iterations $i$ and $j$ we have 
\begin{equation}\label{temp1}
c_{i-1}\not=c_{j-1}.
\end{equation}
Indeed, if $c_{i-1}=c_{j-1}$, then it follows that $P_{i-1}=P_{j-1}$ and $mode_{i-1}=mode_{j-1}$. 	

Moreover, for $k\in\{i-1,j-1\}$ we have $c_k\in P_*$ iff $mode_k=\text{fw}$. It follows from this that $c_k\not\in N(P_*\setminus \{a,b\})$ whenever $mode_k=\text{bw}$. Observe that in light of part (a) we also have
\begin{equation*}
\text{bw}\in\{mode_k: k\in\{i-1,j-1\}\},
\end{equation*}
for otherwise we would have $c_{i-1}=c_{j-1}$, contradicting \eqref{temp1}. Furthermore, since either $c_{i-1}\in \cone(c_{j-1},a,b)$ or $c_{j-1}\in \cone(c_{i-1},a,b)$, we also have 
\begin{equation*}
\text{fw}\in\{mode_k: k\in\{i-1,j-1\}\}.
\end{equation*}
Hence, we may assume that
\begin{equation}
mode_{i-1}=\text{fw}\qquad \&\qquad mode_{j-1}=\text{bw}.
\end{equation}
Note that we have
\begin{equation}
c_{j-1}\in \cone(c_{i-1},a,b)\qquad \text{or}\qquad c_{i-1}\in \cone(c_{j-1},a,b).
\end{equation}
But $c_{j-1}\in \cone(c_{i-1},a,b)$ (resp. $c_{i-1}\in \cone(c_{j-1},a,b)$) implies in iteration $i$ (resp. $j$) the algorithm would choose $\rev(\mathbf{e_{j-1}})$ (resp. $\rev(\mathbf{e_{i-1}})$) ahead of $\mathbf{e_*}$, contradicting that $state(i)=state_*=state(j)$. Therefore, no two states can be equal. \textbf{(c) }Since the set of possible states is finite, it follows by part (b) that the termination condition $\mathbf{e}=\mathbf{e_0}$ will hold after a finite number of states. \textbf{(d) }For every $i\in \mathbb{N}$ let $\mathscr{P}_{i+1}$ be the union of all first components of the states occurred during an application of the algorithm with $k=i+1$ pebbles%\todo{I do not agree that this way we can formally justify our statement without saying something more. It is not clear that the computation of the algorithm is the same with k pebbles (how the first two pebbles will be distributed buring computation) as with 2 pebbles (say). I am not sure whether you follow but hope you see what I mean. This is true but needs justification.}
 . As every pebbled vertex is joined to $v$ via a path of pebbled vertices, we have $\mathscr{P}_{i+1}\Seq N^{i}[v]$ for each $i$. We use induction on $i$ to show that the reverse inclusion also holds. For $i=2$ it is easy to see that every vertex in $N(v)$ appears in $\mathscr{P}_2$%\todo{Why is this easy to see? One needs to justify that algorithm will always backtrack to v and takes all incident edges one by one. This is true but maybe we need a word here.}
 ; i.e. $\mathscr{P}_2\supseteq N(v)$. Furthermore, since for every $i$ every state occurred using $i+1$ pebbles will also occur using more than $i+1$ pebbles, $\{\mathscr{P}_i\}$ is an increasing sequence with respect to $\Seq$. Now, suppose $\mathscr{P}_{i+1}= N^{i}[v]$ for some $i$ and let $u\in V(G)$ such that $d_G(v,u)=i+1$ and consider a vertex $w\in N(u)\cap N^{i}[v]$. By the induction hypothesis and that $\mathscr{P}_i\Seq \mathscr{P}_{i+1}$, there is an iteration $j$ with state $(P',e',mode')$ in the application of the algorithm with $i+1$ pebbles where $w$ is pebbled for the first time. As such, we will have $mode'=\text{fw}$ and the vertex $\head(\mathbf{e'})$ will have two pebbles in iteration $j$. But then all vertices in $N(w)\setminus (N(P'\setminus\{\head(\mathbf{e'})\})$, including $u$, will be pebbled after iteration $j$ and before $\head(\mathbf{e'})$ is unpebbled.%\todo{{\tiny I got lost here, how this last statement concludes the proof of (d).}} 
\end{proof} %\todo{\tiny As i suggested last time, this is rather simple algorithm -- DFS using k+1 pebbles as marks, where we avoid extend exploration path on vertices that are neighbors of currently pebbled vertices (as these will be pebbled later). Because from the embedding we have the rotation system, the algorithm can determine when to stop before re-exploring the initial arc. How about we describe the algorithm in these lines and also its proof. These are standard techniques and referees may be irritated that instead of doing is this way we are choosing rather formal (almost code) description. Just and idea.}  	

\section{A Traversal Algorithm with Enumeration for Geometric Graphs}\label{alg}
Throughout this section, $G$ is a fixed geometric graph. We shall provide an algorithm for traversal of $G$ using $G\boxdot L$, as defined in Section \ref{sec:virtual_graph}, using a real robot $R^r$ and a virtual robot $R^v$. The virtual robot  $R^v$ essentially does the task of traversing $G\boxdot L$, while $R^r$ moves along edges in $G$ with the objective of simulating $G\boxdot L$ for $R^v$. \\

\noindent The following conditions will be maintained by the algorithm:\begin{itemize}
\item$R^r$ will simulate $R^v$ which will traverse $G\boxdot L$,
\item $R^r$ moves along edges of $G$ and does the task of local exploration of $G$ to ,
\item in the traversal phase, the Euclidean distance between $R^r$ and $R^v$ can be maintained within the constant value of $2\sqrt{2}$,
\item in the exploration phase $R^r$ can efficiently get back to its initial position (at the start of the exploration phase), and
\item the algorithm reports every vertex or edge of $G$ exactly once.
\end{itemize}
The entire algorithm consists of two main phases: the {\em preliminary phase} with the goal of finding the minimum edge $e_{\min}$ of $G\boxdot L$, as defined in \cite{Chavez} for quasi-planar graphs, and the {\em traversal phase} with the goal of traversing $G\boxdot L$ and reporting every vertex and edge of $G$ exactly once. In both phases, $R^v$ runs traversal on $G\boxdot L$ via executing $\succc$. In order for $R^v$ to execute $\succc(\mathbf{e})$ for some arc $\mathbf{e}$, $R^r$ finds $\mathbf{e}$, essentially by running Algorithm \ref{k-th.neigborhood}. Similarly, all other queries are handled by $R^r$.% tha for .  whenever re-positioning $R^r$ is underway, we apply Algorithm \ref{k-th.neigborhood} with the additional task of checking if a newly pebbled vertex is in a face of $L$ that contains the current edge on its boundary, and simply terminate Algorithm \ref{k-th.neigborhood} once such a vertex is found.\\   	
%{\Huge\textcolor{red}{UpToHere}}
\\
\noindent\textbf{Preliminary Phase:}
Since, as one can easily see, $\mathbf{e_{\min}}$ is also the minimum arc of $L_G$,  it suffices to perform the traversal algorithm on $L_G$ until $\mathbf{e_{\min}}$ is found. To this end, in stage 0 of the preliminary phase we put $R^r$ at any vertex $v\in V(G)$, choose the minimum arc of $s_v$ as the starting arc, and place $R^v$ in its head. In every subsequent stage, $R^v$ executes $\succc(\mathbf{e})$ for some arc $\mathbf{e}$, which requires $R^r$ to find $\succc(\mathbf{e})$ and provide it to $R^v$ by running Algorithm \ref{k-th.neigborhood}. Note that traversing  $L_G$ can be carried out in accordance with the traversal algorithm in either\cite{Chavez} or in \cite{bose}. In addition, since the vertices of $L_G$ are all virtual and $R^r$ has to remain in $V(G)$, in no stage of the preliminary phase will $R^r$ and $R^v$ be in the same vertex of $G\boxdot L$. At the end of a stage of the traversal $R^r$ either stays in its current face of $L_G$ or leaves it to a neighboring face, as described below:
\begin{enumerate}
\item If $R^v$ has not left to a new face, $R^r$ does not move. 
\item If $R^v$ has left to a new face on an arc $\mathbf{e}$, then starting from its current vertex, $R^r$ follows Algorithm \ref{k-th.neigborhood} with $k=\lfloor 2\sqrt{2} r{(G)}\rfloor$ until a vertex of $G$ in one of the two faces of $L_G$ that contain $\mathbf{e}$ is found.  
\end{enumerate}
The final stage of the preliminary phase is the one in which $\mathbf{e_{\min}}$ is reached by the traversing robot $R^v$.

\begin{algorithm}[H]
\caption{Traversing a Geometric Graphs $G$}\label{DBGeomTrav}
\textbf{Input:} A graph $G\in \mathscr{G}$ where $\mathscr{G}$ is a class of well-embedded geometric graphs, and a vertex $v\in V(G)$   	

\textbf{Output:} The list of (coordinate of) vertices of $G$
\begin{algorithmic}[1]
\State let $e$ be the lower arc of $s_v$ pointing away from $\cor(s_v)$
\State $\mathbf{e}\gets \mathbf{e_0}$
\State $R^r\gets v$, $R^v\gets \head(\mathbf{e})$
\Repeat 
\Comment{finds $e_{\min}$}
\State $\mathbf{e}\gets \rev(\mathbf{e})$
\State $R^v\gets \head(\mathbf{e})$
\While {$\mathbf{e}\neq\entry_{L_G} (\qface_{L_G} (e) )$}
\State $\mathbf{e}\gets \succc _{L_G}(e)$
\State $R^v\gets \head(\mathbf{e})$
\State $\rep$ $R^r$ if necessary
\EndWhile
\Until {$\mathbf{e}=\mathbf{e_{\min}}$}
\State $p\gets \lefttt(\mathbf{e})-(0,1)$
\Repeat 
\Comment{start the traversal on $G\boxdot L$}
\State $\mathbf{e}\gets \succc (\mathbf{e})$
\State $R^v\gets \head(\mathbf{e})$
\State $\rep$ $R^r$ if necessary
\If {$p\in\cone(\tail(\mathbf{e}),\head(\mathbf{e}),\head(\succc(\mathbf{e})))$}
\State \text{\textbf{report }}{$\tail(\mathbf{e})$}
\State $R^v\gets \head(\mathbf{e})$
\State \text{\textbf{report }} every edge $\tail(\mathbf{e})w\in C(L_G)$ such that in the lexicographic order $(\xcor(\tail(\mathbf{e})),\ycor(\tail(\mathbf{e}))<(\xcor(w),\ycor(w))$
\EndIf
\If {$e\in E(G)$ and ($|\tail(\mathbf{e})p|<|\head(\mathbf{e})p|$ or ($|\tail(\mathbf{e})p|=|\head(\mathbf{e})p|$ and $\overrightarrow{\tail(\mathbf{e})p}<\overrightarrow{\head(\mathbf{e})p}$)}
\State \text{\textbf{report }}{$e$}
\EndIf
\If {$\mathbf{e}=\entry(\qface(\mathbf{e}))$}
\State $\mathbf{e}\gets \rev{(\mathbf{e})}$
\Else
\If {$\rev{(\mathbf{e})}=\entry(\qface(\rev(\mathbf{e}))$}
\State $\mathbf{e}\gets \rev{(\mathbf{e})}$
\EndIf
\EndIf   
\Until {$\mathbf{e}=\mathbf{e_{\min}}$}
\end{algorithmic}
\end{algorithm}
\noindent\textbf{Traversal Phase:}   	

The final stage of the preliminary phase defines stage 0 of the traversal phase. During this phase, as opposed to the preliminary one, $R^v$ traverses the entire graph $G\boxdot L$, but $R^r$ uses the same shadowing rules as in the preliminary phase to get within the Euclidean distance of $\sqrt{2}$ from $R^v$ at the end of each stage. Moreover, whenever $R^v$ gets to a vertex of $G$, $R^r$ runs Algorithm \ref{k-th.neigborhood} until it gets to the same vertex where $R^v$ is positioned. For reporting vertices and edges we use the same criterion as the one described in \cite{Chavez} with the exception that hypothetical vertices, i.e. vertices of $L_G$, and edges, i.e. edges $L_G$ and the ones in $\{\ell_v:v\in V(G)\}$, will not be reported. Finally, an edge $st\in C_L$ with $(\xcor(s),\ycor(s))<(\xcor(t),\ycor(t))$ gets reported whenever $s$ is reported. 
%\textcolor{red}{\section{Validity of The Algorithms}} 	

\section{Minimum Spanning Tree in Bounded Space}\label{sec:example}One of the problems that can be tackled using local traversal algorithms is the problem of finding minimum spanning trees in log space for certain classes of graphs. In this section we present Algorithm \ref{min.sp.tree} for finding the minimum spanning tree for connected graphs straight-line embedded in the plane with an injective edge-weight assignment. Likewise Kruskal's algorithm \cite{Kruskal}, Algorithm \ref{min.sp.tree} considers edges in the increasing order of their weights. Utilizing local traversal algorithms without mark bits, Algorithm \ref{min.sp.tree} has the advantage of not requiring global information such as the order or size of the graph under consideration- all it needs is to keep track of the weight of the ``current" candidate for being included in the minimum spanning tree, and a couple of pebbles for distinguishing between its end vertices.  	

Let $G$ be a connected planar graph straight-line embedded in the plane with an injective edge-weight assignment $w:E(G)\to (0,\infty)$. Algorithm \ref{min.sp.tree} uses two auxiliary functions $\mathbf{next.weight}$ defined on $\{w(e):e\in E(G)\}\cup\{0\}$ and $\mathbf{discard.weight}$, defined on $\{w(e):e\in E(G)\}$. For every $w\in \{w(e): e\in E(G)\}$ we define $\mathbf{next.weight}(w)$ to be 0 if no edge of $G$ has a weight $>w$, and to be $\min\{w(e): e\in E(G)\text{ and } w(e)>w\}$ otherwise. Moreover, we define $\mathbf{discard.weight}(w)$ to be 1 if there is a path in $G$ between the end-points of the edge with weight $w$ that has all of its edge-weights $< w$; otherwise we put $\mathbf{discard.weight}(w)=0$. Each of these functions can be implemented by an adaptation of the traversal algorithm in \cite{bose} starting from any current vertex. As such, $\mathbf{next.weight}(w)$ require a traversal of $G$ with keeping track of the tentative lightest edge having weight $>w$, while $\mathbf{discard.weight}(w)$ requires a, possibly partial, traversal of the component of the current vertex in the graph obtained from $G$ by removing every edge of a weight $\ge w$. 	

\begin{algorithm}
\caption{Finding a Minimum-Weight Spanning Tree $T$ of $G$}\label{min.sp.tree}
\begin{algorithmic}[1]
\State $x\gets \mathbf{next.weight}(0)$
\State $W_T\gets\{x\}$
\While {$x\not=0$}
\If {$\mathbf{next.weight}(x)=0$} \Return $W_T$
\Else
\State $x\gets \mathbf{next.weight}(x)$
\EndIf
\If {$\mathbf{discard.weight}(x)=0$} $W_T\gets W_T\cup \{x\}$
\EndIf
\EndWhile
\end{algorithmic}
\end{algorithm} 	

\subsection{Validity of the Algorithm} 	

\begin{lemma}\label{validity.min.sp}
Let $G$ be a connected graph with an injective edge-weight assignment $w$, and let $T$ be the minimum-weight spanning tree of $G$ resulting from Kruscal's algorithm. Then, for every edge $e$ of $G$ we have $e\not\in E(T)$ iff there is a cycle $C\Seq G$ containing $e$ with every edge of a wight $\le w(e)$.
\end{lemma} 	

\begin{proof}
For every $e\in E(G)$ let $T_e$ be the spanning subgraph of $G$ with $E(T_e)=\{e'\in E(T): w(e')\le w(e)\}$. Moreover, for every $e\in E(G)\setminus E(T)$ let $C_e$ be the fundamental cycle of $e$ with respect to $T$. If $e\in E(G)\setminus E(T)$, then $e\not\in E(T_e)$ and, in view of Kruskal's algorithm, we will also have $C_e-e\Seq T_e$. This establishes the ``only-if" part. For the converse, let $e\in E(C)$ where $C$ is a cycle having every edge of a weight $\le w(e)$, 
and let
\begin{equation*}
W:=C+\left(\sum C_{e'}: e'\in (E(C)\setminus\{e\})\setminus E(T)\right),
\end{equation*}
where the summation is taken in the cycle space of $G$. Note that $E(W)\setminus \{e\}\Seq E(T_e)\Seq E(T)$. Moreover, since for every $e'\in E(C)\setminus\{e\}$ the edges of $C_{e'}$ are lighter than $e$, $e$ is an edge in $W$. Therefore, $e\in E(G)\setminus E(T)$, as desired. (Indeed, $W$ is the fundamental cycle $C_e$ of $e$ with respect to $T$.)
\end{proof} 
%\textcolor{red}{\section{Concluding Remarks}}	

\bibliographystyle{plain}

\begin{thebibliography}{10}

\bibitem{aleliunas79}
Romas Aleliunas, Richard~M Karp, Richard~J Lipton, Laszlo Lovasz, and Charles
  Rackoff.
\newblock Random walks, universal traversal sequences, and the complexity of
  maze problems.
\newblock In {\em Foundations of Computer Science, 1979., 20th Annual Symposium
  on}, pages 218--223. IEEE, 1979.

\bibitem{armoni00}
Roy Armoni, Amnon Ta-Shma, Avi Widgerson, and Shiyu Zhou.
\newblock An o (log (n) 4/3) space algorithm for (s, t) connectivity in
  undirected graphs.
\newblock {\em Journal of the ACM (JACM)}, 47(2):294--311, 2000.

\bibitem{Avis1991}
David Avis and Komei Fukuda.
\newblock A pivoting algorithm for convex hulls and vertex enumeration of
  arrangements and polyhedra.
\newblock In {\em Proceedings of the Seventh Annual Symposium on Computational
  Geometry}, SCG '91, pages 98--104, New York, NY, USA, 1991. ACM.

\bibitem{Berg}
Mark~De Berg, Marc~Van Kreveld, Rene~Van Oostrum, and Mark Overmars.
\newblock Simple traversal of a subdivision without extra storage.
\newblock {\em International Journal of Geographical Information Science},
  11(4):359--373, 1997.

\bibitem{bose}
Prosenjit Bose and Pat Morin.
\newblock An improved algorithm for subdivision traversal without extra
  storage.
\newblock {\em International Journal of Computational Geometry \&
  Applications}, 12(04):297--308, 2002.

\bibitem{Chavez}
Edgar Chavez, {\v{S}}tefan Dobrev, Evangelos Kranakis, Jaroslav Opatrny,
  Ladislav Stacho, and Jorge Urrutia.
\newblock Traversal of a quasi-planar subdivision without using mark bits.
\newblock {\em Journal of Interconnection Networks}, 5(04):395--407, 2004.

\bibitem{Cook}
Stephen~A Cook and Charles~W Rackoff.
\newblock Space lower bounds for maze threadability on restricted machines.
\newblock {\em SIAM Journal on Computing}, 9(3):636--652, 1980.

\bibitem{gold1977}
Chris~M Gold, TD~Charters, and J~Ramsden.
\newblock Automated contour mapping using triangular element data structures
  and an interpolant over each irregular triangular domain.
\newblock In {\em ACM SIGGRAPH Computer Graphics}, volume~11, pages 170--175.
  ACM, 1977.

\bibitem{Kruskal}
Joseph~B Kruskal.
\newblock On the shortest spanning subtree of a graph and the traveling
  salesman problem.
\newblock {\em Proceedings of the American Mathematical society}, 7(1):48--50,
  1956.

\bibitem{nisan92}
N.~Nisan, E.~Szemeredi, and A.~Wigderson.
\newblock Undirected connectivity in o(log/sup 1.5/n) space.
\newblock In {\em Proceedings of the 33rd Annual Symposium on Foundations of
  Computer Science}, SFCS '92, pages 24--29, Washington, DC, USA, 1992. IEEE
  Computer Society.

\bibitem{nisan92a}
Noam Nisan.
\newblock Pseudorandom generators for space-bounded computation.
\newblock {\em Combinatorica}, 12(4):449--461, 1992.

\bibitem{nisan92b}
Noam Nisan.
\newblock $\text{RL}\subseteq\text{SC}$.
\newblock In {\em Proceedings of the Twenty-fourth Annual ACM Symposium on
  Theory of Computing}, STOC '92, pages 619--623, New York, NY, USA, 1992. ACM.

\bibitem{reingold05}
Omer Reingold.
\newblock Undirected st-connectivity in log-space.
\newblock In {\em Proceedings of the Thirty-seventh Annual ACM Symposium on
  Theory of Computing}, STOC '05, pages 376--385, New York, NY, USA, 2005. ACM.

\bibitem{saks99}
Michael Saks and Shiyu Zhou.
\newblock $\text{Bphspace} (s)\subseteq\text{dspace} (s^3/2)$.
\newblock {\em Journal of Computer and System Sciences}, 58(2):376--403, 1999.

\bibitem{savitch70}
Walter~J Savitch.
\newblock Relationships between nondeterministic and deterministic tape
  complexities.
\newblock {\em Journal of computer and system sciences}, 4(2):177--192, 1970.

\bibitem{trifonov05}
Vladimir Trifonov.
\newblock An o(log n log log n) space algorithm for undirected st-connectivity.
\newblock In {\em Proceedings of the Thirty-seventh Annual ACM Symposium on
  Theory of Computing}, STOC '05, pages 626--633, New York, NY, USA, 2005. ACM.

\end{thebibliography}

\end{document}